\documentclass{amsart}
\usepackage{amsmath, amssymb, latexsym, amsthm, mathrsfs} 
\usepackage[all]{xy}

\newcommand{\ed}{

\appendix

\section{$\mathsf{S}_f(\scrA,\scrB)$}

Properties closely related to our $\mathsf{U}_{f}(\scrA,\scrB)$ were considered in the literature.
Consider, for each $f\in\NN$, the following selection hypothesis.

\begin{description}
\item[$\mathsf{S}_{f}(\scrA,\scrB)$] For all $\cU_1,\cU_2,\dots\in\scrA$,
there are finite $\cF_1\sbst\cU_1,\cF_2\sbst\cU_2,\dots$ such that
such that $|\cF_n|\le f(n)$ for all $n$, and $\Union_n\cF_n\in\scrB$.
\end{description}

In \cite{GFTM95, BukCie04} 
it is proved that for each $f\in\NN$, $\mathsf{S}_{f}(\rmO,\rmO)=\sone(\rmO,\rmO)$.
Indeed, by Remark \ref{ttt} we have that for all $\scrA$,
$$\mathsf{S}_{f}(\scrA,\rmO)=\mathsf{U}_n(\scrA,\rmO)=\sone(\scrA,\rmO).$$

A family $\scrB$ of open covers of $X$ is \emph{finitely thick} \cite{strongdiags} if:
\be
\itm If $\cU\in\scrB$ and for each $U\in\cU$:
\begin{quote}
$\cF_U$ is a finite
nonempty family of open sets such that for each $V\in\cF_U$,
$U\sbst V\neq X$,
\end{quote}
then $\Union_{U\in\cU}\cF_U\in\scrB$.
\itm If $\cU\in\scrB$ and $\cV=\cU\cup\cF$ where $\cF$ is finite
and $X\nin\cF$, then $\cV\in\scrB$.\footnote{We will not use Item (2) of the definition of \emph{finitely thick} here.}
\ee
Many families of ``rich'' covers considered in the literature,
including $\rmO,\Omega,\Ga$ \cite{coc1,coc2}, are finitely thick.
Also, for each of these families, each pair of elements has a joint refinement in
the same family.

The case $\scrA=\scrB=\Omega$ of the following theorem was proved in \cite{GFTM95, Val98}.

\bthm\label{jjj}
Assume that each pair of elements of $\scrA$ has a joint refinement in $\scrA$, and $\scrB$ is finitely thick.
For each $f\in\NN$, $\mathsf{S}_{f}(\scrA,\scrB)=\sone(\scrA,\scrB)$.
\ethm
\bpf
As $1\le f(n)$ for all $n$, $\sone(\scrA,\scrB)$ implies $\mathsf{S}_{f}(\scrA,\scrB)$.
To prove the remaining implication, assume that $X$ satisfies $\mathsf{S}_{f}(\scrA,\scrB)$.

Let $\cU_1,\cU_2,\dots\in\scrA(X)$.
Let $s(n)=f(1)+f(2)+\dots+f(n)$ for all $n$.
For each $n$, take $\cV_n\in\scrA(X)$ refining $\cU_1,\dots,\cU_{s(n)}$.

Apply $\mathsf{S}_{f}(\scrA,\scrB)$ to the sequence $\cV_1,\cV_2,\dots$,
to obtain $\cF_1\sbst\cV_{1},\cF_{2}\sbst\cV_{2},\dots$, such that $|\cF_n|\le f(n)$ for
all $n$, and $\Union_n\cF_n\in\scrB(X)$.

Fix $n$. For each $k\in\{s(n-1)+1,\dots,s(n)\}$, pick $U_k\in\cU_k$ such that each member of $\cF_n$ is contained in
some $U_k$. As $\scrB$ is finitely thick, $\{U_k : k\in\N\}\in\scrB(X)$.
\epf

Thus, in our context, the scheme $\mathsf{S}_{f}(\scrA,\scrB)$ does not introduce new properties.
As we have seen in the present paper, this is not the case for $\mathsf{U}_{f}(\scrA,\scrB)$.

\end{document}}

\newcommand{\splus}{\!+\!}

\newcommand{\Setting}[7]{\xymatrix@R=4pt@C=7pt{#1\ar@{-}[r]&#2\ar@{-}[r]&#3\\&#4\ar@{-}[u]\\
#5\ar@{-}[uu]\ar@{-}[r] & #6\ar@{-}[u]\ar@{-}[r] & #7\ar@{-}[uu]}}

\newcommand{\intvl}[2]{{[#1(#2),\allowbreak #1(#2\splus1))}}

\newcommand{\bq}{\begin{quote}}
\newcommand{\eq}{\end{quote}}
\newcommand{\cl}[1]{\overline{#1}}
\newcommand{\CH}{the Continuum Hypothesis}

\newcommand{\inv}{^{-1}}
\newcommand{\Cantor}{{\{0,1\}^\N}}

\newcommand{\Ga}{\Gamma}

\newcommand{\N}{\mathbb{N}}
\newcommand{\NN}{{\N^{\N}}}

\newcommand{\roth}{{[\N]^{\oo}}}
\newcommand{\Fin}{{[\N]^{<\oo}}}

\newcommand{\sseq}[1]{\{#1 : n\in\N\}}

\newcommand{\op}{\operatorname}

\newcommand{\scrA}{\mathscr{A}}
\newcommand{\scrB}{\mathscr{B}}

\newcommand{\cB}{\mathcal{B}}

\newcommand{\cF}{\mathcal{F}}

\newcommand{\rmO}{\mathrm{O}}

\newcommand{\Q}{\mathbb{Q}}
\newcommand{\R}{\mathbb{R}}
\newcommand{\cU}{\mathcal{U}}
\newcommand{\Union}{\bigcup}
\newcommand{\cV}{\mathcal{V}}

\newcommand{\Impl}{\Rightarrow}
\long\def\forget#1\forgotten{}

\newcommand{\fb}{\mathfrak{b}}
\newcommand{\fc}{\mathfrak{c}}
\newcommand{\fd}{\mathfrak{d}}

\newcommand{\oo}{\infty}

\newcommand{\x}{\times}

\newcommand\comp{^{\text{\tt c}}}
\newcommand{\nin}{\notin}

\newcommand{\sbst}{\subseteq}
\newcommand{\spst}{\supseteq}
\newcommand{\sm}{\setminus}
\newcommand{\as}{\subseteq^*}

\newcommand{\non}{\op{non}}

\newtheorem{thm}{Theorem}[section]
\newcommand{\bthm}{\begin{thm}} \newcommand{\ethm}{\end{thm}}
\newtheorem{prop}[thm]{Proposition}
\newcommand{\bprp}{\begin{prop}} \newcommand{\eprp}{\end{prop}}
\newtheorem{fact}[thm]{Fact}
\newcommand{\bfct}{\begin{fact}} \newcommand{\efct}{\end{fact}}
\newtheorem{lem}[thm]{Lemma}
\newcommand{\blem}{\begin{lem}} \newcommand{\elem}{\end{lem}}
\newtheorem{claim}[thm]{Claim}
\newcommand{\bclm}{\begin{claim}} \newcommand{\eclm}{\end{claim}}
\newtheorem{cor}[thm]{Corollary}
\newcommand{\bcor}{\begin{cor}} \newcommand{\ecor}{\end{cor}}

\newtheorem{conj}[thm]{Conjecture}
\newcommand{\bcnj}{\begin{conj}} \newcommand{\ecnj}{\end{conj}}
\newtheorem{prob}[thm]{Problem}
\newcommand{\bprb}{\begin{prob}} \newcommand{\eprb}{\end{prob}}

\theoremstyle{definition}
\newtheorem{defn}[thm]{Definition}
\newcommand{\bdfn}{\begin{defn}} \newcommand{\edfn}{\end{defn}}
\theoremstyle{remark}
\newtheorem{rem}[thm]{Remark}
\newcommand{\brem}{\begin{rem}} \newcommand{\erem}{\end{rem}}
\newtheorem{cnv}[thm]{Convention}
\newcommand{\bcnv}{\begin{cnv}} \newcommand{\ecnv}{\end{cnv}}
\newtheorem{exam}[thm]{Example}
\newcommand{\bexm}{\begin{exam}} \newcommand{\eexm}{\end{exam}}
\newcommand{\bpf}{\begin{proof}} \newcommand{\epf}{\end{proof}}
\newcommand{\be}{\begin{enumerate}}
\newcommand{\ee}{\end{enumerate}}
\newcommand{\bi}{\begin{itemize}}
\newcommand{\itm}{\item}
\newcommand{\ei}{\end{itemize}}

\forget
\forgotten

\newcommand{\sone}{\mathsf{S}_1}
\newcommand{\sfin}{\mathsf{S}_\mathrm{fin}}
\newcommand{\ufin}{\mathsf{U}_\mathrm{fin}}

\title[Menger and Hurewicz Problems]{Menger's and Hurewicz's Problems: Solutions from ``The Book'' and refinements}
\author{Boaz Tsaban}
\address{Department of Mathematics, Bar-Ilan University, Ramat Gan 52900, Israel}
\email{tsaban@math.biu.ac.il}
\urladdr{http://www.cs.biu.ac.il/\~{}tsaban}



\begin{document}

\begin{abstract}
We provide simplified solutions of Menger's and Hure\-wicz's problems
and conjectures, concerning generalizations of $\sigma$-compactness.
The reader who is new to this field will find a self-contained treatment
in Sections 1, 2, and 5.

Sections 3 and 4 contain new results, based on the mentioned simplified solutions.
The main new result is that there are concrete uncountable sets of reals $X$ (indeed, $|X|=\fb$),
which have the following property:
\begin{quote}
Given point-cofinite covers $\cU_1,\cU_2,\dots$ of $X$, there are for each $n$ sets $U_n,V_n\in\cU_n$,
such that each member of $X$ is contained in all but finitely many of the sets
$U_1\cup V_1,U_2\cup V_2,\dots$
\end{quote}
This property is strictly stronger than Hurewicz's covering property.
Miller and the present author showed that
one cannot prove the same result if we are only allowed to pick one set from each $\cU_n$.
\end{abstract}

\maketitle

\centerline{\emph{Dedicated to Professor Gideon Schechtman}}

\tableofcontents

\forget
{\sf Todo?
\be
\itm Dispose of Lemma \ref{ksp}? For this, need a direct proof that $\fc$-concentrated
is not $\sigma$-compact.
\itm E.g., in the dichotomic proofs we may just use an unbounded set of cardinality $\fb$
in the case $\fb<\fd$.
\itm Sierpi\'nski's cute argument that Luzin sets are not Hurewicz:
Fix a countable dense subset $\sseq{a_n}$ of $\R\sm L$.
For all $n,m$, let $$U^n_m=\R\sm \cl{B}_\frac{1}{m}(a_n).$$
The sets $U^n_m$ increase with $m$, and
$L\sbst \Union_m U^n_m$ for each $n$.
If $L$ was Hurewicz, then $L\sbst \Union_n U^n_{m_n}$ for appropriate
$m_n$. As $\Union_n B_{1/m_n}(a_n)$ is open dense, $\Union_n U^n_{m_n}$ is meager,
a contradiction.
\ee
}
\forgotten

\section{Menger's Conjecture}

In 1924, Menger \cite{Menger24} introduced the following
basis property for a metric space $X$:
\begin{quote}
For each basis $\cB$ for the topology of $X$, there are $B_1,B_2,\dots\in\cB$ such that $\lim_{n\to\infty}\op{diam}(B_n) = 0$, and $X=\Union_nB_n$.
\end{quote}
Soon thereafter, Hurewicz \cite{Hure25} observed\footnote{Hurewicz stated his observation
without proof. A proof was provided by Lelek in \cite[Theorem 2]{Lelek}.}
that Menger's basis
property can be reformulated as follows:
\begin{quote}
For all given open covers $\cU_1,\cU_2,\dots$ of $X$,
there are finite $\cF_1\sbst\cU_1,\cF_2\sbst\cU_2,\dots$ such that $\Union_n\cF_n$ is a cover of $X$.
\end{quote}
We introduce some convenient notation, suggested by Scheepers in \cite{coc1}.
We say that $\cU$ is a \emph{cover} of $X$ if $X=\Union\cU$,\footnote{We follow the set theoretic standard that,
for a family of sets $\cF$, $\Union\cF$ means the union of all elements of $\cF$.}
but $X\nin\cU$.
Let $X$ be a topological space, and $\scrA,\scrB$ be families of covers of $X$.
We consider the following statements.
\begin{description}
\item[$\sone(\scrA,\scrB)$] For all $\cU_1,\cU_2,\dots\in\scrA$, there are
$U_1\in\cU_1,U_2\in\cU_2,\dots$ such that $\sseq{U_n}\in\scrB$.
\item[$\sfin(\scrA,\scrB)$] For all $\cU_1,\cU_2,\dots\in\scrA$, there are
finite $\cF_1\sbst\cU_1,\cF_2\sbst\cU_2,\dots$ such that $\Union_n\cF_n\in\scrB$.
\item[$\ufin(\scrA,\scrB)$] For all $\cU_1,\cU_2,\dots\in\scrA$, none containing
a finite subcover, there are finite $\cF_1\sbst\cU_1,\cF_2\sbst\cU_2,\dots$ such that $\sseq{\Union\cF_n}\in\scrB$.
\end{description}
Let $\rmO(X)$ be the family of all open covers of $X$.
We say that $X$ satisfies $\sone(\rmO,\rmO)$ if the statement $\sone(\rmO(X),\rmO(X))$ holds.
This way, $\sone(\rmO,\rmO)$ is a property of topological spaces.
A similar convention applies to all properties of this type.

Hurewicz's observation tells that for metric spaces, Menger's basis property is equivalent to
$\sfin(\rmO,\rmO)$. This is a natural generalization of compactness.
Note that indeed, every \emph{$\sigma$-compact} space (a countable union of compact spaces) satisfies $\sfin(\rmO,\rmO)$.
Menger made the following conjecture.

\bcnj[Menger \cite{Menger24}]
A metric space $X$ satisfies $\sfin(\rmO,\rmO)$ if, and only if, $X$ is $\sigma$-compact.
\ecnj
Hurewicz proved that when restricted to analytic spaces, Menger's Conjecture is true.

Recall that a set $M\sbst\R$ is \emph{meager} (or \emph{of Baire first category})
if $M$ is a union of countably many nowhere dense sets.
A set $L\sbst\R$ is a \emph{Luzin set} if $L$ is uncountable, and
for each meager set $M$, $L\cap M$ is countable.

Luzin sets can be constructed assuming \CH{}: Every meager set is contained in a Borel (indeed, $F_\sigma$) meager set. Let $M_\alpha$, $\alpha<\aleph_1$ be all Borel meager sets.
For each $\alpha<\aleph_1$,
take $x_\alpha\in\R\sm\Union_{\beta<\alpha}M_\beta$. Then $L=\{x_\alpha : \alpha<\aleph_1\}$ is a Luzin set.

A subset of $\R$ is \emph{perfect} if it is nonempty, closed, and has no isolated points.
In \cite{Hure27}, Hurewicz quotes an argument of Sierpi\'nski, proving (more than) the following.

\bthm[Sierpi\'nski]\label{LM}
Every Luzin set satisfies $\sfin(\rmO,\rmO)$, and is not $\sigma$-compact.
\ethm
\bpf
Let $\cU_1,\cU_2,\dots$ be open covers of a Luzin set $L\sbst\R$.
Let $D=\sseq{d_n}$ be a dense subset of $L$. For each $n$, pick $U_n\in\cU_n$ such that $d_n\in U_n$.
Let $U=\Union_nU_n$. Then $L\sm U$ is nowhere dense, and thus countable.
Enumerate $L\sm U=\sseq{x_n}$.
For each $n$, pick $V_n\in\cU_n$
such that $x_n\in V_n$. Then $L\sm U\sbst\Union_nV_n$, and thus $\sseq{U_n,V_n}$ is a cover of $L$, with at most two elements from each $\cU_n$.\footnote{The interested reader may wish to show in a similar manner that actually,
every Luzin set satisfies $\sone(\rmO,\rmO)$. We will not use this fact.}

The following short argument for the remaining assertion was suggested to us by
Vadim Kulikov.
Assume that $L=\Union_nK_n$, a countable union of compact sets.
Then, for each $n$, $\R\sm K_n$ is open (since $K_n$ is closed) and dense (since $K_n\sbst L$).
Thus, $K_n$ is nowhere dense, and therefore $L$ is meager; a contradiction.
\epf

Thus, Menger's Conjecture is settled if one assumes \CH{}.
In 1988, Fremlin and Miller \cite{FM88} settled Menger's Conjecture in ZFC.
They used the concept of a scale, which we now define.
This concept is normally defined using $\NN$, but for our purposes it is
easier to work with $P(\N)$ (this will become clear later).

Let $P(\N)$ be the family of all subsets of $\N$, and $\Fin,\roth\sbst P(\N)$ denote
the family of all finite subsets of $\N$ and the family of all infinite subsets of $\N$,
respectively.
For $a\in\roth$ and $n\in\N$, $a(n)$ denotes the $n$-th element in the increasing enumeration of $a$.

For $a,b\in\roth$, let $a\le^* b$ mean: $a(n)\le b(n)$ for all but finitely many $n$.
A subset $Y$ of $\roth$ is \emph{dominating} if for each $a\in\roth$ there is $b\in Y$
such that $a\le^* b$. Let $\fd$ denote the minimal cardinality of a dominating subset
of $\roth$.
A \emph{scale} is a dominating set $S\sbst\roth$, which has a $\le^*$-increasing enumeration
$S=\{s_\alpha : \alpha<\fd\}$, that is, such that $s_\alpha\le^* s_\beta$ for all $\alpha<\beta<\fd$.

Scales require special hypotheses to be constructed.
Indeed, say that a subset $Y$ of $\roth$ is \emph{unbounded} if it is unbounded with respect to
$\le^*$, that is, for each $a\in\roth$ there is $b\in Y$ such that $b\not\le^* a$.
Let $\fb$ denote the minimal cardinality of an unbounded subset of $\roth$.
$\fb\le\fd$, and strict inequality is consistent. (Indeed, $\fb<\fd$ holds in the Cohen real model.)

\blem[folklore]\label{scalebd}
There is a scale if, and only if, $\fb=\fd$.
\elem
\bpf
$(\Leftarrow)$ Let $\{d_\alpha : \alpha<\fb\}\sbst\roth$ be dominating. For each $\alpha<\fb$,
choose $s_\alpha$ to be a $\le^*$-bound of $\{d_\beta,s_\beta : \beta<\alpha\}$.

$(\Impl)$ Let $S=\{s_\alpha : \alpha<\fd\}$ be a scale, and assume that $\fb<\fd$.
Let $\{b_\alpha : \alpha<\fb\}\sbst\roth$ be unbounded.
For each $\alpha$, take $\beta_\alpha<\fd$
such that $b_\alpha\le^* s_{\beta_\alpha}$.

Let $c\in\roth$ witness that $\{s_{\beta_\alpha} : \alpha<\fb\}$ is not dominating,
and let $\gamma<\fd$ be such that $c\le^* s_\gamma$.
For each $\alpha<\fb$, $s_\gamma\not\le^* s_{\beta_\alpha}$, and thus $b_\alpha\le^* s_{\beta_\alpha}\le^* s_\gamma$.
Thus, $\{b_\alpha : \alpha<\fb\}$ is bounded. A contradiction.
\epf

The canonical way to construct sets of reals from scales (more generally, from subsets of $P(\N)$)
is as follows.
$P(\N)$ is identified with Cantor's space $\Cantor$, via characteristic functions.
This defines the canonical topology on $P(\N)$.
Cantor's space is homeomorphic to the canonical middle-third Cantor set $C\sbst[0,1]$,
and the homeomorphism is (necessarily, uniformly) continuous in both directions.
Thus, subsets of $P(\N)$ exhibiting properties preserved by taking (uniformly)
continuous images may be converted into subsets of $[0,1]$ with the same properties.
We may thus work in $P(\N)$.

The \emph{critical cardinality} of a (nontrivial) property $P$ of set of reals, denoted $\non(P)$,
is the minimal cardinality of a set of reals $X$ such that $X$ does not have the property $P$.
The following is essentially due to Hurewicz \cite{Hure27}.

\blem[folklore]\label{md}
$\non(\sfin(\rmO,\rmO))=\fd$.
\elem
\bpf
$(\ge)$ Let $X$ be a set of reals with $|X|<\fd$. Let $\cU_1,\cU_2,\dots$ be open covers of $X$.
Since $X$ is Lindel\"of, we may assume that these covers are countable,
and enumerate them $\cU_n=\{U^n_m : m\in\N\}$.

Define for each $x\in X$ a set $a_x\in\roth$
by
$$a_x(n)=\min\{m>a_x(n-1) : x\in U^n_1\cup U^n_2\cup\dots\cup U^n_m\}.$$
As $|\{a_x : x\in X\}|<\fd$, there is (in particular)
$c\in\roth$ such that for each $x\in X$, $a_x(n)\le c(n)$ for some $n$.
Take $\cF_n=\{U^n_1,\dots,U^n_{c(n)}\}$ for
all $n$. Then $\Union_n\cF_n$ is a cover of $X$.

$(\le)$ Let $D$ be a dominating subset of $\roth$.
Consider the open covers $\cU_n=\{U^n_m : m\in\N\}$, $n\in\N$, where
$$U^n_m=\{a\in\roth : a(n)=m\}.$$
For all finite $\cF_1\sbst\cU_1,\cF_2\sbst\cU_2,\dots$, there is $x\in D$ such that for all but finitely many $n$,
$x(n)>\max\{m : U^n_m\in\cF_n\}$ (and thus $x\nin\Union\cF_n$).

But if $X$ satisfies $\sfin(\rmO,\rmO)$, then for all open covers $\cU_1,\cU_2,\dots$ of $X$, there are
finite $\cF_1\sbst\cU_1,\cF_2\sbst\cU_2,\dots$, such that for each $x\in X$, $x$ belongs to
$\Union\cF_n$ for infinitely many $n$: To see this, split the given sequence $\cU_1,\cU_2,\dots$ into infinitely many disjoint subsequences,
and apply $\sfin(\rmO,\rmO)$ to each of these subsequences separately.

Thus, dominating subsets of $\roth$ do not satisfy $\sfin(\rmO,\rmO)$.
\epf

Let $\kappa$ be an infinite cardinal.
A set of reals $X$ is \emph{$\kappa$-concentrated} on a set $Q$ if,
for each open set $U$ containing $Q$, $|X\sm U|<\kappa$.

\begin{lem}[folklore \cite{sfh}]\label{concen}
Assume that a set of reals $X$ is $\fc$-concentrated on a countable set $Q$.
Then $X$ does not contain a perfect set.
\end{lem}
\begin{proof}
Assume that $X$ contains a perfect set $P$.
Then $P\sm Q$ is Borel and uncountable.
A classical result of Alexandroff tells that every uncountable
Borel set contains a perfect set. Let $C\sbst P\sm Q$ be a perfect set.\footnote{As $Q$ is countable, one can
alternatively prove directly that $P\sm Q$ contains a perfect set.}
Then $U=\R\sm C$ is open and contains $Q$, and $C=P\sm U\sbst X\sm U$ has cardinality $\fc$.
Thus, $X$ is not $\fc$-concentrated on $Q$.
\end{proof}

\blem[Cantor--Bendixon]\label{ksp}
Every uncountable $\sigma$-compact set $X\sbst\R$ contains a perfect set.
\elem
\bpf
By moving to a subset, we may assume that $X$ is an uncountable compact, and thus
closed, set. By the Cantor--Bendixon Theorem, $X$ contains a perfect set.
\epf

\bthm[Fremlin--Miller \cite{FM88}]\label{FMThm}
Menger's Conjecture is false.
\ethm
\bpf
As perfect sets of reals have cardinality continuum, we have by Lemma \ref{ksp}
that if $\fb<\fd$, then any set of reals of cardinality $\fb$ is a counter-example.

Thus, assume that $\fb=\fd$ (this is the interesting case), and let $S=\{s_\alpha : \alpha<\fd\}\sbst\roth$ be a scale
(Lemma \ref{scalebd}).

$S\cup\Fin$ satisfies $\sfin(\rmO,\rmO)$: This is similar to the argument about Luzin sets
satisfying $\sfin(\rmO,\rmO)$. Given open covers $\cU_1,\cU_2,\dots$ of $S\cup\Fin$, take
$U_1\in\cU_1,U_2\in\cU_2,\dots$, such that $\Fin\sbst\Union_nU_n$. We can do that because
$\Fin$ is countable. Let $U=\Union_nU_n$. $P(\N)\sm U$ is closed and thus compact.
For each $n$, the evaluation map $e_n:\roth\to\N$ defined by $e_n(a)=a(n)$ is continuous.
Thus, $e_n[P(\N)\sm U]$ is compact and thus finite, for all $n$.
Therefore, there is a $\le^*$-bound $b$ for $P(\N)\sm U$. Take $\alpha<\fd$ such that $b<^* s_\alpha$.
Then
$$S\sm U = S\cap (P(\N)\sm U) \sbst \{s_\beta : \beta<\fd, s_\beta\le^* b\}\sbst \{s_\beta : \beta<\alpha\}$$
has cardinality $<\fd$, and thus satisfies $\sfin(\rmO,\rmO)$. Let $\cF_1\sbst\cU_1,\cF_2\sbst\cU_2,\dots$
be such that $S\sm U\sbst\Union_n\cF_n$. Then $S\cup\Fin\sbst\Union_n\cF_n\cup\{U_n\}$.

$S\cup\Fin$ is not $\sigma$-compact: We have just seen that it is $\fd$-concentra\-ted on the countable
set $\Fin$. Use Lemmata \ref{ksp} and \ref{concen}.
\epf

A reader not familiar with dichotomic proofs may be perplexed by the proof of the Fremlin--Miller Theorem \ref{FMThm}.
It gives a ZFC result by considering an undecidable statement.
Indeed, it shows that there is a certain set of reals, but does not tell us what this set is
(unless we know in advance whether $\fb<\fd$ or $\fb=\fd$). Another way to view this is as follows.

Sets of reals $X$ satisfying $P$ because $|X|<\non(P)$ are in a sense \emph{trivial} examples for this property.
From this point of view, the real question is, given a property $P$, whether there are
sets of reals of cardinality at least $\non(P)$, which satisfy $P$. The proof of Theorem \ref{FMThm}
answers this in the positive only when $\fb=\fd$.
However, with a small modification we get a complete answer.

\bdfn
A \emph{$\fd$-scale} is a dominating set $S=\{s_\alpha : \alpha<\fd\}\sbst\roth$, such that for all $\alpha<\beta<\fd$, $s_\beta\not\le^* s_\alpha$.
\edfn

\blem
There are $\fd$-scales.
\elem
\bpf
Let $\{d_\alpha : \alpha<\fd\}\sbst\roth$ be dominating. For each $\alpha<\fd$,
choose $s_\alpha$ to be a witness that $\{s_\beta : \beta<\alpha\}$ is not dominating, such that
in addition, $d_\alpha\le^* s_\alpha$.
\epf

An argument similar to that in the proof of Theorem \ref{FMThm} gives the following.

\blem
Every $\fd$-scale is $\fd$-concentrated on $\Fin$.\qed
\elem

We therefore have the following.

\bthm[Bartoszy\'nski--Tsaban \cite{ideals}]\label{dscale}
For each $\fd$-scale $S$, $S\cup\Fin$ satisfies $\sfin(\rmO,\rmO)$, and is not $\sigma$-compact.
In other words, $S\cup\Fin$ is a counter-example to Menger's Conjecture.\qed
\ethm

Theorem \ref{dscale} is generalized in Tsaban--Zdomskyy \cite{sfh}.

We conclude the section with some easy improvements of statements made above.

Define the following subfamily of $\rmO(X)$:
$\cU\in\Ga(X)$ if $\cU$ is infinite, and each element of $X$ is contained in all but
finitely many members of $\cU$. If $\cU\in\Ga(X)$, then every infinite subset of $\cU$ belongs
to $\Ga(X)$. Thus, we may assume for our purposes that elements of $\Ga(X)$ are countable.

\bcor[Just, et al.\ \cite{coc2}]\label{jj}
$\sone(\Gamma,\rmO)$ implies $\sfin(\rmO,\rmO)$.
\ecor
\bpf
Let $X$ be a set of reals satisfying $\sone(\Gamma,\rmO)$, and let $\cU_1,\cU_2,\dots\in\rmO(X)$.
The claim is trivial if some $\cU_n$ contains a finite subcover. Thus, assume that this is not the case.

As sets of reals are Lindel\"of, we may assume that each $\cU_n$ is countable, say $\cU_n=\{U^n_m : m\in\N\}$.
Let
$$\cV_n=\left\{\Union_{k\le m}U^n_k : m\in\N\right\}.$$
Then $\cV_n\in\Ga(X)$. Applying  $\sone(\Gamma,\rmO)$ there are $m_n$, $n\in\N$, such that
$\sseq{\Union_{k\le m_n}U^n_k}$ is a cover of $X$. For each $n$, the finite sets $\cF_n=\{U^n_k : k\le m_n\}\sbst\cU_n$
are as required in the definition of $\sfin(\rmO,\rmO)$.
\epf

A modification of the proof of Lemma \ref{md} yields the following.

\blem[Just, et al.\ \cite{coc2}]
$\non(\sone(\Gamma,\rmO))=\fd$.\qed
\elem
\bpf
By Corollary \ref{jj} and Lemma \ref{md},
$$\non(\sone(\Gamma,\rmO))\le\non(\sfin(\rmO,\rmO))=\fd.$$
To prove the remaining inequality, assume that $|X|<\fd$, and $\cU_1,\cU_2,\allowbreak\dots\in\Ga(X)$.
We may assume that for each $n$, $\cU_n$ is countable, and enumerate it $\cU_n=\{U^n_m : m\in\N\}$.
For each $x\in X$, let
$$a_x(n)=\min\{k>a_x(n-1) : (\forall m\ge k)\ x\in U^n_m\}$$
for all $n$. (In the case $n=1$, omit the restriction $k>a_x(n-1)$.)
$|\{a_x : x\in X\}|<\fd$. Let $d\in\roth$ exemplify that $\{a_x : x\in X\}$ is not dominating,
and take $\cF_n=\{U^n_1,\dots,U^n_{d(n)}\}$. Then each $x\in X$ belongs to $\Union\cF_n$
for infinitely many $n$.
\epf

\bcor\label{dcon}
Each set which is $\fd$-concentrated on a countable subset, satisfies $\sone(\Ga,\rmO)$.\qed
\ecor

\bcor[Bartoszy\'nski--Tsaban \cite{ideals}]\label{MenSol}
For each $\fd$-scale $S$, $S\cup\Fin$ satisfies $\sone(\Ga,\rmO)$.\qed
\ecor

$\sone(\Ga,\rmO)$ is strictly stronger that $\sfin(\rmO,\rmO)$.
While every $\sigma$-compact set satisfies the latter, we have the following.

\blem[Just, et al.\ \cite{coc2}]\label{noperf}
If $X$ satisfies $\sone(\Ga,\rmO)$, then $X$ has no perfect subsets.
\elem
\bpf
We give Sakai's proof \cite[Lemma 2.1]{Sakai07}. Assume that $X$ has a perfect subset and satisfies $\sone(\Ga,\rmO)$.
Then $X$ has a subset $C$ homeomorphic to Cantor's space $\Cantor$. $C$ is compact, and thus
closed in $X$, and therefore satisfies $\sone(\Ga,\rmO)$ as well.\footnote{It is easy to see that all properties involving
open covers, considered in this paper, are hereditary for closed subsets \cite{coc2}.}
Thus, it suffices to show that $\Cantor$ does not satisfy $\sone(\Ga,\rmO)$.
We show instead that its homeomorphic copy $(\Cantor)^\N$ does not satisfy $\sone(\Ga,\rmO)$.

Let $C_1,C_2,\dots$ be pairwise disjoint nonempty clopen subsets of $\Cantor$.
Let $U_1,U_2,\dots$ be the complements of $C_1,C_2,\dots$, respectively.
For each $n$, let $\pi_n:(\Cantor)^\N\to\Cantor$ be the projection on the $n$-th coordinate.
Then $\cU_n=\{\pi_n\inv[U_m] : m\in\N\}\in\Gamma(X)$ for all $n$.
But for all $\pi_1\inv[U_{m_1}]\in\cU_1,\pi_2\inv[U_{m_2}]\in\cU_2,\dots$, we have
that $\Pi_nC_n$ is disjoint of $\Union_n\pi_n\inv[U_{m_n}]$.
\epf

\section{Hurewicz's Conjecture}

Hurewicz suspected that Menger's Conjecture was false.
For this reason, he introduced in \cite{Hure25} a formally stronger property, which in our
notation is $\ufin(\rmO,\Ga)$. It is easy to see that every $\sigma$-compact set satisfies, in
fact, $\ufin(\rmO,\Ga)$, and analogously to Menger, Hurewicz made the following.

\bcnj[Hurewicz \cite{Hure25}]
A metric space $X$ satisfies $\ufin(\rmO,\Ga)$ if, and only if, $X$ is $\sigma$-compact.
\ecnj

The following easy fact is instructive.

\blem\label{moresadd}
$X$ satisfies $\ufin(\rmO,\Gamma)$ if, and only if, for all $\cU_1,\cU_2,\dots$, none having a finite subcover of $X$,
there is a decomposition $X=\Union_kX_k$, such that for each $k$, there are finite subsets $\cF^k_1\sbst\cU_1,\cF^k_2\sbst\cU_2,\dots$,
such that for each $x\in X_k$, $x\in\Union\cF^k_n$ for all but finitely many $n$.
\elem
\bpf
For each $n$, take $\cF_n=\Union_{k\le n}\cF^k_n$. Then $\sseq{\Union\cF_n}\in\Ga(X)$.
\epf


$S\sbst\R$ is a \emph{Sierpi\'nski set} if $S$ is uncountable, and
for each Lebesgue measure zero set $N$, $S\cap N$ is countable.
Since every perfect set contains a perfect set of Lebesgue measure zero,
a Sierpi\'nski set cannot contain a perfect subset, and therefore is not $\sigma$-compact (Lemma \ref{ksp}).
A construction similar to that of a Luzin set described above, shows that
\CH{} implies the existence of Sierpi\'nski sets.
We do not know when the following observation was made first.

\bthm[folklore]\label{SH}
Every Sierpi\'nski set satisfies $\ufin(\rmO,\Ga)$.
\ethm
\bpf
The following proof is a slightly simplified version of the one given in \cite{coc2}.

Let $S$ be a Sierpi\'nski set. $S=\Union_nS\cap[-n,n]$, and thus by Lemma \ref{moresadd},
we may assume that the outer measure $p$ of $S$ is finite.
Since $S$ is Sierpi\'nski, $p>0$.\footnote{Otherwise, $S$ would have measure zero, and thus be countable.}
Let $B\spst S$ be a Borel set of measure $p$.

Let $\cU_1,\cU_2,\dots$ be open covers of $S$. We may assume that each $\cU_n$ is countable,
and enumerate $\cU_n=\{U^n_m : m\in\N\}$. We may assume that all $U^n_m$ are Borel subsets of $B$.
For each $n$, $\Union_m U^n_m\spst S$, and thus has measure $p$ for each $n$.
Thus, for each $N$ there is $f_N\in\NN$ such that $\Union_{k=1}^{f_N(n)} U^n_k$ has measure $\ge (1-1/2^{n+N})p$,
and consequently, $A_N=\bigcap_n\Union_{k=1}^{f_N(n)} U^n_k$ has measure $\ge (1-1/2^N)p$.

Then $A=\Union_NA_N$ has measure $p$, and thus $S\sm A$ is countable.
The countable decomposition $S=(S\sm A)\cup \Union_N A_N$ is as required in Lemma \ref{moresadd},
by the countability of $S\sm A$ and the definition of $A_N$.
\epf

A stronger statement can be proved in a similar manner.

\bthm[Just, et al.\ \cite{coc2}]
Every Sierpi\'nski set satisfies $\sone(\Ga,\Ga)$ (even when we consider Borel covers instead of open ones).
\ethm
\bpf
Replace, in the proof of Theorem \ref{SH}, $U^n_m$ by $\bigcap_{k\ge m}U^n_k$.
Let $f\in\NN$ be such that for each $x\in S\sm A$, $x\in \bigcap_{k\ge f(n)}U^n_k$ for all but finitely many $n$.
Let $g$ be a $\le^*$-bound of $\{f_N : N\in\N\}\cup\{f\}$. Then the choice $U^1_{g(1)}\in\cU_1,U^2_{g(2)}\in\cU_2,\dots$ is as required.
\epf

Thus, \CH{} implies the failure of Hurewicz's Conjecture.
A complete refutation, however, was only discovered in 1996, by Just, Miller, Scheepers, and Szeptycki,
in their seminal paper \cite{coc2}.

\bthm[Just, et al.\ \cite{coc2}]
Hurewicz's Conjecture is false.
\ethm

We will not provide the full solution from \cite{coc2} here (since we provide a simpler one),
but just discuss its main ingredients. The argument in \cite{coc2} is dichotomic.
Recall that $\fb$ is the minimal cardinality of a set $B\sbst\roth$ which is unbounded with respect to $\le^*$.
A proof similar to that of Lemma \ref{md} gives the following two results, which are also essentially due to
Hurewicz \cite{Hure27}.

\blem[folklore]\label{bnh}
An unbounded subset of $\roth$ cannot satisfy $\ufin(\rmO,\Ga)$.\qed
\elem

\blem[folklore]\label{hb}
$\non(\sone(\Ga,\Ga))=\non(\ufin(\rmO,\Ga))=\fb$.\qed
\elem

Thus, if $\fb>\aleph_1$ then any set of cardinality $\aleph_1$ is a counter-example to Hurewicz's Conjecture.


\bdfn
A \emph{$\fb$-scale} is an unbounded set $\{b_\alpha : \alpha<\fb\}\sbst\roth$, such that
the enumeration is increasing with respect to $\le^*$ (i.e., $b_\alpha\le^* b_\beta$ whenever
$\alpha<\beta<\fb$).
\edfn

Like $\fd$-scales, $\fb$-scales can be constructed without special hypotheses.

\blem[folklore]\label{bsc}
There are $\fb$-scales.
\elem
\bpf
Let $\{x_\alpha : \alpha<\fb\}\sbst\roth$ be unbounded. For each $\alpha<\fb$,
choose $b_\alpha$ to be a $\le^*$-bound of $\{b_\beta : \beta<\alpha\}$, such that $x_\alpha\le^* b_\alpha$.
\epf

The argument in \cite{coc2} proceeds as follows. We have just seen that the case $\fb>\aleph_1$ is trivial.
Thus, assume that $\fb=\aleph_1$. Then there is a $\fb$-scale $B=\{b_\alpha : \alpha<\fb\}\sbst\roth$ such that
in addition, for all $\alpha<\beta<\fb$, $b_\beta\sm b_\alpha$ is finite.\footnote{We will not use this fact here,
but here is a proof: Fix an unbounded family $\{x_\alpha : \alpha<\fb\}\sbst\roth$.
At step $\alpha$, we have a countable set $B_\alpha=\{b_\beta : \beta<\alpha\}$ such that
for all $\gamma<\beta<\fb$, $b_\beta\sm b_\gamma$ is finite. In particular, each finite subset of $B_\alpha$
has an infinite intersection. Enumerate $B_\alpha=\sseq{s_n}$, and for each $n$ pick $m_n\in s_1\cap\dots\cap s_n$
such that $m_n>m_{n-1}$. Let $c$ be a $\le^*$-bound of $B_\alpha$, and let
$b_\alpha$ be a subset of $\sseq{m_n}$, such that $\max\{c,x_\alpha\}\le^* b_\alpha$.}
It is proved in
\cite{coc2} that for such $B$, $B\cup\Fin$ satisfies $\ufin(\rmO,\Ga)$.
An argument similar to the one given in Theorem \ref{FMThm} for scales shows the following.

\blem
Every $\fb$-scale $B$ is $\fb$-concentrated on $\Fin$. In particular, $B\cup\Fin$ is not $\sigma$-compact.\qed
\elem

Unfortunately, the existence of $\fb$-scales as in the proof of \cite{coc2} is undecidable.
This is so because Scheepers proved that for this type of $\fb$-scales, $B\cup\Fin$ in fact satisfies $\sone(\Gamma,\Gamma)$ \cite{alpha_i}
(see also \cite{BBC}), and we have the following.

\bthm[Miller--Tsaban \cite{BBC}]\label{miltsa}
It is consistent that for each set of reals satisfying $\sone(\Ga,\Ga)$, $|X|<\fb$.
Indeed, this is the case in Laver's model.
\ethm

Bartoszy\'nski and Shelah have discovered an ingenious direct solution to Hurewicz's Conjecture,
which can be reformulated as follows.

\bthm[Bartoszy\'nski--Shelah \cite{BaSh01}]\label{BShThm}
For each $\fb$-scale $B$, $B\cup\Fin$ satisfies $\ufin(\rmO,\Ga)$.
\ethm

We provide a simplified proof of this theorem, using a method of Galvin and Miller from \cite{GM}.
For natural numbers $n,m$, let $[n,m)=\{n,n+1,\dots,m-1\}$.

\blem[folklore]\label{bddslalom}
Let $Y\sbst\roth$. The following are equivalent:
\be
\itm $Y$ is bounded;
\itm There is $s\in\roth$ such that for each $a\in Y$, $a\cap\intvl{s}{n}\neq\emptyset$ for all
but finitely many $n$.
\ee
\elem
\begin{proof}
$(1\Impl 2)$ Let $b\in\roth$ be a $\le^*$-bound for $Y$.
Define inductively $s\in\roth$ by
\begin{eqnarray*}
s(1) & = & b(1)\\
s(n+1) & = & b(s(n))+1
\end{eqnarray*}
For each $a\in Y$ and all but finitely many $n$, $s(n)\le a(s(n))\le b(s(n))<s(n+1)$, that
is, $a(s(n))\in\intvl{s}{n}$.

$(2\Impl 1)$ Let $s$ be as in $(2)$.
$s$ has countably many cofinite subsets.
Let $b\in\roth$ be a $\le^*$-bound of all cofinite subsets of $s$.
Let $a\in Y$ and choose $n_0$ such that for each
$n\ge n_0$, $a\cap\intvl{s}{n}\neq\emptyset$.
Choose $m_0$ such that $a(m_0)\in\intvl{s}{n_0}$.
By induction on $n$, we have that $(a(n)\le) a(m_0+n)\le s(n_0+1+n)$
for all $n$. For large enough $n$, we have that $s(n_0+1+n)\le b(n)$,
thus $a\le^* b$.
\end{proof}

\blem[Galvin--Miller \cite{GM}]\label{GMlem}
Assume that $\Fin\sbst X\sbst P(\N)$.
For each $\cU\in\Ga(X)$,\footnote{Less than that is required of the given covers. See the proof.}
there are $a\in\roth$ and distinct $U_1,U_2,\dots\in\cU$,
such that for each $x\sbst\N$, $x\in U_n$ whenever $x\cap\intvl{a}{n}=\emptyset$.
\elem
\bpf
Let $a(1)=1$. For each $n\ge 1$: As $\cU\in\Ga(X)$, each finite subset of $X$ is contained in
infinitely many elements of $\cU$. Take $U_n\in\cU\sm\{U_1,\dots,U_{n-1}\}$,
such that $P([1,a(n)))\sbst U_n$.
As $U_n$ is open, for each $s\sbst [1,a(n))$ there is $k_s$ such that
for each $x\in P(\N)$ with $x\cap[1,k_s)=s$, $x\in U_n$. Let $a(n+1)=\max\{k_s : s\sbst[1,a(n))\}$.
\epf

Given the methods presented thus far, the following proof boils down to the fact that,
if we throw fewer than $n$ balls into $n$ bins, at least one bin remains empty.

\bpf[Proof of Theorem \ref{BShThm}]
Let $B=\{b_\alpha : \alpha<\fb\}$ be a $\fb$-scale. Let $\cU_1,\cU_2,\dots$ be open covers of $B\cup\Fin$.
By the argument in the proof of Corollary \ref{jj}, we may assume that each $\cU_n$ is a point-cofinite
cover of $B\cup\Fin$.

For each $n$, take $a_n$ and distinct $U^n_1,U^n_2,\dots$ for $\cU_n$ as in Lemma \ref{GMlem}.
We may assume that $a_n(1)=1$.
Let $\alpha$ be such that $I=\{n : a_n(n+1)<b_\alpha(n)\}$ is infinite.
As $|\{b_\beta : \beta<\alpha\}|<\fb$, $\{b_\beta : \beta<\alpha\}$ satisfies $\sone(\Ga,\Ga)$ (Lemma \ref{hb}).
Thus, there are $m_n$, $n\in I$, such that $\{U^n_{m_n} : n\in I\}\in\Ga(\{b_\beta : \beta<\alpha\})$.
Take $\cF_n=\emptyset$ for $n\nin I$, and $\cF_n=\{U^n_1,\dots,U^n_n\}\cup\{U^n_{m_n}\}$ for
$n\in I$.

As $\sseq{\Union\cF_n}=\{\Union\cF_n : n\in I\}\cup\{\emptyset\}$,
it suffices to show that for each $x\in X$, $x\in\Union\cF_n$ for all but finitely many $n\in I$.
If $x\in\Fin$, then for each large enough $n\in I$, $x\cap \intvl{a_n}{n}=\emptyset$ (because
$a_n(n)\ge n$), and thus $x\in U^n_{n}\in\cF_n$.
For $\beta<\alpha$, $b_\beta\in U^n_{m_n}\sbst\Union\cF_n$ for all large enough $n$.

For $\beta\ge\alpha$ (that's the interesting case!)
and all but finitely many $n\in I$, $b_\beta(n)\ge b_\alpha(n)>a_n(n+1)$. Thus, $|b_\beta\cap [1,a_n(n+1))|<n$.
As $[1,a_n(n+1))=\Union_{i=1}^{n} \intvl{a_n}{i}$ is a union of $n$ intervals,
there must be $i\le n$ such $b_\beta\cap\intvl{a_n}{i}=\emptyset$, and thus $b_\beta\in U^n_i\sbst\Union\cF_n$.
\epf

A multidimensional version of the last proof gives the following.

\bthm[Bartoszy\'nski--Tsaban \cite{ideals}]
For each $\fb$-scale $B$, all finite powers of the set $B\cup\Fin$ satisfy $\ufin(\rmO,\Ga)$.\qed
\ethm

Indeed, Zdomskyy and the present author proved in \cite{sfh} that any finite product
$(B_1\cup\Fin)\x\dots\x(B_1\cup\Fin)$, with $B_1,\dots,B_k$ $\fb$-scales, satisfies $\ufin(\rmO,\Ga)$.

In a work in progress, the method introduced here is used to prove the following,
substantially stronger, result.

\bthm[Miller--Tsaban--Zdomskyy]
For each $\fb$-scale $B$ and each set of reals $H$ satisfying $\ufin(\rmO,\Ga)$, $(B\cup\Fin)\x H$ satisfies $\ufin(\rmO,\Ga)$.
\ethm

\section{Strongly Hurewicz sets of reals, in ZFC}

Consider, for each $f\in\NN$, the following selection hypothesis.

\begin{description}
\item[$\mathsf{U}_{f}(\scrA,\scrB)$] For all $\cU_1,\cU_2,\dots\in\scrA$, none containing
a finite subcover, there are finite $\cF_1\sbst\cU_1,\cF_2\sbst\cU_2,\dots$ such that
such that $|\cF_n|\le f(n)$ for all $n$, and $\sseq{\Union\cF_n}\in\scrB$.
\end{description}

\brem
One may require in the definition of $\mathsf{U}_{f}(\scrA,\scrB)$ that each $\cF_n$ is nonempty.
This will not change the property when $\scrA,\scrB\in\{\rmO,\Ga\}$, since we may
assume that the given covers get finer and finer. This can be generalized to most types of
covers considered in the field.
\erem

$\mathsf{U}_{f}(\scrA,\scrB)$ depends only on $\limsup_nf(n)$.

\blem\label{anyggoes}
Assume that for each $\cV\in\scrB$, we have $\{\emptyset\}\cup\cV\in\scrB$, and 
every cofinite subset of $\cV$ is also in $\scrB$.
For all $f,g\in\NN$ with $\limsup_n f(n)=\limsup_n g(n)$,
$\mathsf{U}_{f}(\scrA,\scrB)=\mathsf{U}_{g}(\scrA,\scrB)$.
\elem
\bpf
The argument is as in the proofs of \cite[3.2--3.5]{GFTM95} and \cite[Lemma 3]{Val98a},
concerning similar concepts in other contexts.

Let $\cU_1,\cU_2,\dots\in\scrA(X)$.
Let $m_1<m_2<\dots$ be such that $f(n)\le g(m_n)$ for all but finitely many $n$.
Apply $\mathsf{U}_{f}(\scrA,\scrB)$ to the sequence $\cU_{m_1},\cU_{m_2},\dots$,
to obtain $\cF_{m_1}\sbst\cU_{m_1},\cF_{m_2}\sbst\cU_{m_2},\dots$, such that $|\cF_{m_n}|\le f(n)$ for
all $n$, and $\sseq{\Union\cF_{m_n}}\in\scrB(X)$. For $k\nin\sseq{m_n}$ we can take $\cF_k=\emptyset$.
Then $\sseq{\Union\cF_n}=\{\emptyset\}\cup\sseq{\Union\cF_{m_n}}\in\scrB(X)$, and $|\cF_n|\le g(n)$ for all but finitely many $n$. Changing finitely many additional sets $\cF_n$
to $\emptyset$, we have $|\cF_n|\le g(n)$ for all $n$. 
\epf

Thus, for each $f\in\NN$ with $\limsup_nf(n)=\oo$, $\mathsf{U}_{f}(\scrA,\scrB)=\mathsf{U}_{\mathrm{id}}(\scrA,\allowbreak\scrB)$,
where $\mathrm{id}$ is the identity function,  $\mathrm{id}(n)=n$ for all $n$. We henceforth use the notation
$$\mathsf{U}_{n}(\scrA,\scrB)$$
for $\mathsf{U}_{\mathrm{id}}(\scrA,\scrB)$.

Our proof of Theorem \ref{BShThm} shows the following.

\bthm\label{main}
For each $\fb$-scale $B$, $B\cup\Fin$ satisfies $\mathsf{U}_n(\Ga,\Ga)$.
\ethm
\bpf
In the proof of Theorem \ref{BShThm} we show that $B\cup\Fin$ satisfies $\mathsf{U}_{n+1}(\Ga,\Ga)$.
By Lemma \ref{anyggoes}, this is the same as $\mathsf{U}_n(\Ga,\Ga)$.
\epf

We will soon show that $\mathsf{U}_n(\Ga,\Ga)$ is strictly stronger than $\ufin(\rmO,\Ga)$.

A cover $\cU$ of $X$ is \emph{multifinite} \cite{strongdiags} if there exists a
partition of $\cU$ into infinitely many finite covers of $X$.
Let $\scrA$ be a family of covers of $X$. $\gimel(\scrA)$
is the family of all covers $\cU$ of $X$ such that:
Either $\cU$ is multifinite, or there exists a partition $\mathcal{P}$ of $\cU$
into finite sets such that
$\{\bigcup\mathcal{F}:\mathcal{F}\in\mathcal{P}\}\setminus\{X\}\in\scrA$ \cite{GlCovs}.

The special case $\gimel(\Ga)$ was first studied by Ko\v{c}inac and Scheepers \cite{coc7},
where it was proved that $\ufin(\rmO,\Ga)=\sfin(\Omega,\gimel(\Ga))$.
Additional results of this type are available in
Babinkostova--Ko\v{c}inac--Scheepers \cite{coc8}, and in general form in Samet--Scheepers--Tsaban \cite{GlCovs}.

\bthm[Samet, et al.\ \cite{GlCovs}]\label{rduc}
$\ufin(\Ga,\gimel(\Ga))=\sfin(\Ga,\allowbreak\gimel(\Ga))$.
\ethm

\bthm\label{ungimel}
$\mathsf{U}_n(\Ga,\Ga)$ implies $\sone(\Ga,\gimel(\Ga))$.
\ethm
\bpf
We prove the following, stronger statement: Assume that $X$ satisfies $\mathsf{U}_n(\Ga,\Ga)$, and let $s(n)=1+\dots+n=(n+1)n/2$.
For all $\cU_1,\cU_2,\dots\in\Ga(X)$, there are $U_1\in\cU_1,U_2\in\cU_2,\dots$, such that for each $x\in X$,
$x\in \Union_{k=s(n)}^{s(n+1)}U_k$ for all but finitely many $n$.

Let $\cU_1,\cU_2,\dots\in\Ga(X)$. We may assume that for each $n$, $\cU_{n+1}$ refines $\cU_n$.
Apply $\mathsf{U}_n(\Ga,\Ga)$ to $\cU_{s(1)},\cU_{s(2)},\dots$ to obtain
$U_1\in\cU_{s(1)},U_2,U_3\in\cU_{s(2)},\dots$, such that for each $x\in X$,
$x\in\Union_{k=s(n)+1}^{s(n+1)}U_k$ for all but finitely many $n$.
For each $n$ and each $k=s(n)+1,\dots,s(n+1)$, replace $U_k$ by an equal or larger set from $\cU_k$.
\epf

\brem
The statement at the beginning of the last proof is in fact a characterization of $\mathsf{U}_n(\Ga,\Ga)$.
\erem

\brem\label{ttt}
In general, if every pair of elements of $\scrA$ has a joint refinement in $\scrA$, and $\scrB$ is finitely thick in the sense
of \cite{strongdiags}, then $\mathsf{U}_n(\scrA,\scrB)$ implies $\sone(\scrA,\gimel(\scrB))$.

In particular, when $\scrB=\rmO$, $\gimel(\scrB)=\rmO$, and thus $\mathsf{U}_n(\scrA,\rmO)=\sone(\scrA,\rmO)$.
For example, $\mathsf{U}_n(\Ga,\rmO)=\sone(\Ga,\rmO)$.
\erem

Thus, the Bartoszy\'nski--Shelah Theorem tells that for each $\fb$-scale $B$, $B\cup\Fin$ satisfies
$\sfin(\Ga,\gimel(\Ga))$, whereas Theorem \ref{main} tells that it indeed satisfies $\sone(\Ga,\gimel(\Ga))$.
As $\ufin(\rmO,\Ga)$ does not even imply $\sone(\Ga,\rmO)$ (Lemma \ref{noperf}),
we have that $\mathsf{U}_n(\Ga,\Ga)$ is strictly stronger than $\ufin(\rmO,\Ga)$.

\bthm[Tsaban--Zdomskyy \cite{HurPerf}]\label{nogg}
Assume \CH{} (or just $\fb=\fc$).
There is a $\fb$-scale $B$ such that no set of reals containing $B\cup\Fin$ satisfies $\sone(\Ga,\Ga)$.
\ethm

By Theorems \ref{main} and \ref{nogg}, $\mathsf{U}_n(\Ga,\Ga)\neq\sone(\Gamma,\Gamma)$.
Thus, $\mathsf{U}_n(\Ga,\Ga)$ is strictly in between $\sone(\Gamma,\Gamma)$ and $\ufin(\rmO,\Ga)$.

A natural refinement of the Problem 9, solved in Theorem \ref{nogg}, is the following.

\bprb[Zdomskyy]
Is there a set of reals $X$ without perfect subsets, such that $X$ satisfies $\ufin(\rmO,\Ga)$
but not $\mathsf{U}_n(\Ga,\Ga)$?
\eprb


\section{A visit at the border of ZFC}

By Lemma \ref{anyggoes}, there are only the following kinds of (strongly) Hurewicz properties:
$\ufin(\Ga,\Ga)$, $\mathsf{U}_n(\Ga,\Ga)$, and $\mathsf{U}_c(\Ga,\Ga)$, for constants $c\in\N$.
For $c=1$, $\mathsf{U}_c(\Ga,\Ga)=\sone(\Ga,\Ga)$, and thus by
the results of the previous section, at least three of these properties are distinct. (We
consider properties \emph{distinct} if they are not provably equivalent.)

By Theorem \ref{miltsa}, $\mathsf{U}_1(\Ga,\Ga)$ may be trivial. The next strongest property
is $\mathsf{U}_2(\Ga,\Ga)$. We prove that it is not trivial.

\bdfn
Let $s,a\in\roth$.
$s$ \emph{slaloms}\footnote{Short for ``is a slalom for''.} $a$ if
$a\cap\intvl{s}{n}\neq\emptyset$ for all but finitely many $n$.
$s$ \emph{slaloms} a set $Y\sbst\roth$ if it slaloms each $a\in Y$.
\edfn

By Lemma \ref{bddslalom}, a set $Y\sbst\roth$ is bounded if, and only
if, there is $s$ which slaloms $Y$.

\bdfn
A \emph{slalom $\fb$-scale} is an unbounded set $\{b_\alpha : \alpha<\fb\}\sbst\roth$,
such that $b_\beta$ slaloms $b_\alpha$ for all $\alpha<\beta<\fb$.
\edfn

By Lemma \ref{bddslalom}, we have the following.

\blem\label{bscip}
There are slalom $\fb$-scales.\qed
\elem

We are now ready to prove the main result of this paper.

\bthm\label{main2}
For each slalom $\fb$-scale $B$, $B\cup\Fin$ satisfies $\mathsf{U}_2(\Ga,\Ga)$.
\ethm
\bpf
Let $B=\{b_\alpha : \alpha<\fb\}$ be a slalom $\fb$-scale. Let $\cU_1,\cU_2,\dots\in\Ga(B\cup\Fin)$.

For each $n$, take $a_n\in\roth$ and distinct $U^n_1,U^n_2,\dots$ for $\cU_n$ as in Lemma \ref{GMlem}.
We may assume that $a_n(1)=1$.
Let $a\in\roth$ slalom $\{a_n : n\in\N\}$.
As $B$ is unbounded, there is by Lemma \ref{bddslalom} $\alpha<\fb$,
such that $I=\{m : [a(m),a(m\splus3))\cap b_\alpha=\emptyset\}$
is infinite. (Otherwise, $\sseq{a(3n)}$ would slalom $B$.)
For each $n$, let
$$I_n=\{m\ge n : [a_n(m),a_n(m\splus2))\cap b_\alpha=\emptyset\}.$$
As $a$ slaloms $a_n$, $I_n$ is infinite, and therefore $\{U^n_m : m\in I_n\}\in\Ga(B\cup\Fin)$.

As $|\{x_\beta : \beta<\alpha\}|<\fb$, $\{x_\beta : \beta<\alpha\}$ satisfies $\sone(\Ga,\Ga)$ (Lemma \ref{hb}),
and thus, there are $m_n\in I_n$, $n\in\N$, such that $\{U^n_{m_n} : n\in\N\}\in\Ga(\{x_\beta : \beta<\alpha\})$.
We claim that
$$\{U^n_{m_n}\cup U^n_{m_n+1} : n\in\N\}\in\Ga(B\cup\Fin).$$
If $x\in\Fin$, then for each large enough $n$, $x\cap \intvl{a_n}{m_n}=\emptyset$ (because $m_n\ge n$),
and thus $x\in U^n_{m_n}$.
For $\beta<\alpha$, $b_\beta\in U^n_{m_n}$ for all large enough $n$, by the choice of $m_n$.

For $\beta\ge\alpha$ (that's the interesting case), we have the following:
Let $m_n\in I_n$, and let $k$ be such that
$$b_\alpha(k)<a_n(m_n)<a_n(m_n+2)\le b_\alpha(k+1).$$
If $n$ is large, then $k$ is large, and as $b_\beta$ slaloms $b_\alpha$, there is $i$ such that
$$b_\beta(i)\le b_\alpha(k)<a_n(m_n)<a_n(m_n+2)\le b_\alpha(k+1)<b_\beta(i+2).$$
There are two possibilities for $a_n(m_n+1)$:
If $a_n(m_n\splus1)\le b_\beta(i+1)$, then $\intvl{a_n}{m_n}\cap b_\beta=\emptyset$, and thus $b_\beta\in U^n_{m_n}$.
Otherwise, $a_n(m_n\splus1)>b_\beta(i+1)$, and thus $[a_n(m_n\splus1),a_n(m_n\splus2))\cap b_\beta=\emptyset$.
Therefore, $b_\beta\in U^n_{m_n+1}$ in this case.
\epf

\bthm
Assume \CH{} (or just $\fb=\fc$).
There is a slalom $\fb$-scale $B$ such that $B\cup\Fin$ satisfies $\mathsf{U}_2(\Ga,\Ga)$, but
no set of reals containing $B\cup\Fin$ satisfies $\sone(\Ga,\Ga)$.
\ethm
\bpf
Consider the proof of Theorem \ref{nogg}, given in \cite{HurPerf}.
We need only make sure that in Proposition 2.5 of \cite{HurPerf},
$B$ can be constructed in a way that it is a \emph{slalom} $\fb$-scale.
This should be taken care of in the second paragraph of page 2518.

At step $\alpha<\fb$ of this construction, we are given a set $Y$ with $|Y|=|\alpha|<\fb$,
and a set $a_\alpha\in\roth$. Take an infinite $b_\alpha\sbst a_\alpha$ such
that $b_\alpha$ slaloms $Y$. (E.g., take a slalom $b$ for $Y$, and then define $b_\alpha\sbst a_\alpha$
by induction on $n$, such that for each $n$, $|b\cap\intvl{b_\alpha}{n}|\ge 2$.)
By induction on $n$,
thin out $b_\alpha$ such that it
satisfies the displayed inequality there for all $n$. $b_\alpha$ remains a slalom for
$Y$.

Theorem \ref{main2} guarantees that $B\cup\Fin$ satisfies $\mathsf{U}_2(\Ga,\Ga)$.
\epf

By Theorem \ref{miltsa}, it is consistent that $\sone(\Ga,\Ga)$ is trivial,
whereas by Theorem \ref{main2}, $\mathsf{U}_2(\Ga,\Ga)$ is never trivial.
The following remains open.

\bcnj
$\mathsf{U}_2(\Ga,\Ga)$ is strictly stronger than $\mathsf{U}_n(\Ga,\Ga)$.
\ecnj

\section{The Hurewicz Problem}

In the same 1927 paper Hurewicz asked the following.

\bprb[Hurewicz \cite{Hure27}]
Is there a metric space satisfying $\sfin(\rmO,\allowbreak\rmO)$, but not $\ufin(\rmO,\Ga)$?
\eprb

In a footnote added at the proof stage (the same one mentioned before
Theorem \ref{LM}), Hurewicz quotes the following, which solves his problem if \CH{} is assumed.

\bthm[Sierpi\'nski]\label{hhh}
Every Luzin set satisfies $\sfin(\rmO,\rmO)$, but not $\ufin(\rmO,\Ga)$.
\ethm

\bpf
Let $L$ be a Luzin set. We have already proved that $L$ satisfies $\sfin(\rmO,\rmO)$ (Theorem \ref{LM}).
It remains to show that $L$ does not satisfy $\ufin(\rmO,\Ga)$.

As $L$ contains no perfect sets, $\R\sm L$ is dense in $\R$. Fix a countable dense $D\sbst\R\sm L$.
$\R\sm D$ is homeomorphic to $\R\sm\Q$,\footnote{$D$ is order-isomorphic to $\Q$. An order
isomorphism $f:D\to\Q$ extends uniquely to and order isomorphism $f:\R\to\R$ by setting $f(r)=\sup\{f(d) : d<r\}$.
The restriction of $f$ to $\R\sm D$ is a homeomorphism.}
which in turn is homeomorphic to $\roth$ (e.g., using continued fractions).

As $L\sbst\R\sm D$, we may assume that $L\sbst\roth$.\footnote{If $L$ is a Luzin set in a topological space
$X$ and $f:X\to Y$ is a homeomorphism, then $f[L]$ is a Luzin set in $Y$, since ``being meager''
is preserved by homeomorphisms.}
By Lemma \ref{bnh}, it suffices to show that $L$ is unbounded.
For each $b\in\roth$, the set
$$\{a\in\roth : a\le^* b\}=\Union_{n\in\N}\{a\in\roth : (\forall m\ge n)\ a(m)\le b(m)\},$$
with each $\{a\in\roth : (\forall m\ge n)\ a(m)\le b(m)\}$ nowhere dense. Thus,
$\{a\in\roth : a\le^* b\}$ is meager, and therefore does not contain $L$.
\epf

Hurewicz's problem remained, however, open until the end of 2002.

\bthm[Chaber--Pol \cite{ChaPol}]\label{CPThm}
There is a set of reals satisfying $\sfin(\rmO,\allowbreak\rmO)$ but not $\ufin(\rmO,\Ga)$.
\ethm

Chaber and Pol's proof is topological and uses a technique due to Michael.
The following combinatorial proof contains the essence of their proof.

\bpf[Proof of Theorem \ref{CPThm}]
The proof is dichotomic. If $\fb<\fd$, then any unbounded $B\sbst\roth$ of cardinality $\fb$
satisfies $\sfin(\rmO,\rmO)$ (Lemma \ref{md}) but not $\ufin(\rmO,\Ga)$ (Lemma \ref{bnh}).

\blem\label{hh}
For each $s\in\roth$, there is $a\in\roth$ such that: $a\comp=\N\sm a\in\roth$,
$a\not\le^* s$, and $a\comp\not\le^* s$.
\elem
\bpf
Let $m_1>s(1)$. For each $n>1$, let $m_n>s(m_{n-1})$. Let $a=\Union_n[m_{2n-1},m_{2n})$.
For each $n$:
\begin{eqnarray*}
a(m_{2n}) & \ge &  m_{2n+1}>s(m_{2n});\\
a\comp(m_{2n-1}) & \ge &  m_{2n}>s(m_{2n-1}).\qedhere
\end{eqnarray*}
\epf

So, assume that $\fb=\fd$.
Fix a scale $\{s_\alpha : \alpha<\fd\}\sbst\roth$.
For each $\alpha<\fd$, use Lemma \ref{hh} to pick $a_\alpha\in\roth$ such that:
\be
\itm $a_\alpha\comp=\N\sm a_\alpha$ is infinite;
\itm $a_\alpha\not\le^* s_\alpha$; and
\itm $a_\alpha\comp\not\le^* s_\alpha$.
\ee
Let $A=\{a_\alpha : \alpha<\fd\}$. For $b\in\roth$, let $\alpha<\fd$ be such that
$b<^* s_\alpha$. Then $\{\beta : a_\beta\le^* b\}\sbst\alpha$.
As in the proof of Theorem \ref{FMThm}, this implies that
$A$ is $\fd$-concentrated on $\Fin$, and thus $A\cup\Fin$ satisfies $\sfin(\rmO,\rmO)$
(indeed, $\sone(\Gamma,\rmO)$---Corollary \ref{dcon}).

On the other hand, $A\cup\Fin$ is homeomorphic to
$Y=\{x\comp : x\in A\cup\Fin\}$, which is an unbounded subset of $\roth$ (by item (3) of the construction).
By Lemma \ref{bnh}, $Y$ (and therefore $A\cup\Fin$) does not satisfy $\ufin(\rmO,\Ga)$.
\epf

The advantage of the last proof is its simplicity. However, it does not provide an
explicit example, and in the case $\fb<\fd$ gives a trivial example, i.e., one of cardinality smaller than
$\non(\sfin(\rmO,\rmO))$.
We conclude with an explicit solution.

\bthm[Tsaban--Zdomskyy \cite{sfh}]\label{TsZd}
There is a set of reals of cardinality $\fd$, satisfying $\sfin(\rmO,\rmO)$ (indeed, $\sone(\Gamma,\rmO)$),
but not $\ufin(\rmO,\Ga)$.
\ethm

Our original proof uses in its crucial step a topological argument.
Here, we give a more combinatorial argument, based on a (slightly
amended) lemma of Mildenberger.

A set $Y\sbst\roth$ is \emph{groupwise dense}
if:
\be
\itm $a\as y\in Y$ implies $a\in Y$; and
\itm For each $a\in\roth$, there is an infinite $I\sbst\N$ such that
$\Union_{n\in I}\allowbreak\intvl{a}{n}\in Y$.
\ee
For $Y$ satisfying $(1)$, $Y$ is groupwise dense if, and only if, $Y$ is nonmeager \cite{BlassHBK}.

\bpf[Proof of Theorem \ref{TsZd}]
Fix a dominating set $\{d_\alpha : \alpha<\fd\}$.
Define $a_\alpha\in\roth$ by induction on $\alpha<\fd$.
Step $\alpha$: Let $Y=\{d_\beta,a_\beta : \beta<\alpha\}$. $|Y|<\fd$.

The following is proved by  Mildenberger as part of the proof of \cite[Theorem 2.2]{Mild01},
except that we eliminate the ``next'' function from her argument.

\blem[Mildenberger \cite{Mild01}]
For each $Y\sbst\roth$ with $|Y|<\fd$, $G=\{a\in\roth : (\forall y\in Y)\ a\not\le^* y\}$ is groupwise
dense.
\elem
\bpf
Clearly, $G$ satisfies (1) of the definition of groupwise density. We verify (2).

We may assume that $Y$ is closed under maxima of finite subsets. Let $g\in\roth$ be
a witness that $Y$ is not dominating. Then the family of all sets $\{n : y(n)<g(n)\}$, $y\in Y$,
can be extended to a nonprincipal ultrafilter $\cU$.

Let $a\in\roth$. By thinning out $a$, we may assume that $g(a(n))<a(n+1)$ for all $n$.
For $i=0,1,2$, let
$$a_i=\Union_{n\in\N}[a(3n+i),a(3n+i+1)).$$
Then there is $i$ such that $a_i\in\cU$. We claim that $a_{i+2 \bmod 3}\in G$.
Let $y\in Y$. For each $k$ in the infinite set $\{n : y(n)<g(n)\}\cap a_i$, let
$n$ be such that $k\in [a(3n\splus i),a(3n\splus i\splus 1))$. Then
$$y(k)<g(k)<g(a(3n+i+1))<a(3n+i+2)\le a_{i+2 \bmod 3}(k),$$
because $a(3n\splus i\splus2)$ is the first element of $a_{i\splus2 \bmod 3}$ greater or equal to
$k$, and $a_{i+2 \bmod 3}(k)\ge k$.
\epf

Let $G=\{a\in\roth : (\forall y\in Y)\ a\not\le^* y\}$. As $G$ is groupwise dense,
there is $a_\alpha\in G$ such that $a_\alpha\comp$ is infinite and $a_\alpha\comp\not\le^* d_\alpha$.
To see this, take an interval partition as in the proof of Lemma \ref{hh}. Then there is
an infinite subfamily of the even intervals, whose union $a_\alpha$ is in $G$.
For each $n$ such that $[m_{2n-1},m_{2n})\sbst a_\alpha$,
$a\comp(m_{2n-1}) \ge m_{2n}>s(m_{2n-1})$.\footnote{Alternatively,
note that $\{a : a\comp\le^* d_\alpha\}$ is homeomorphic to the meager
set $\{a : a\le^* d_\alpha\}$, and thus cannot contain a groupwise dense (i.e., nonmeager) set.}

Thus, there is
$$a_\alpha\in \{a\in\roth : (\forall y\in Y)\ a\not\le^* y\}\sm \{a\in\roth : a\comp\le^* d_\alpha\}.$$
Continue exactly as in the above proof of Theorem \ref{CPThm}.
\epf

Chaber and Pol's Theorem in \cite{ChaPol} is actually stronger than Theorem \ref{CPThm} above,
and establishes the existence of a set of reals
$X$ such that $X$ does not satisfy $\ufin(\rmO,\Ga)$,\footnote{And thus neither any finite power of $X$, since $X$ is
a continuous image of $X^k$ for each $k$.}
but all finite powers of $X$ satisfy $\sfin(\rmO,\rmO)$.

Their proof shows that if $\fb=\fd$, then there is such
an example of cardinality $\fd$. The assumption ``$\fb=\fd$'' was weakened to ``$\fd$ is regular''
by Tsaban and Zdomskyy \cite{sfh}, but the following remains open.

\bprb
Is there, provably in ZFC, a nontrivial (i.e., one of cardinality at least $\fd$)
example of a set of reals such that $X$ does not satisfy $\ufin(\rmO,\Ga)$, but
all finite powers of $X$ satisfy $\sfin(\rmO,\rmO)$?
\eprb

In other words, the question whether there is a nondichotomic proof of Chaber and Pol's full theorem remains open.

\forget
\section{Order, please?}

The distinction between the various properties considered here is quite subtle, and led to quite a few confusions,
in the published as well as the unpublished literature. Let us try put order in these matters. By the
results presented and surveyed here, we have the following:

Fremlin--Miller's solution of Menger's Conjecture also gives, when $\fb=\fd$,
a counter-example for Hurewicz's Conjecture (because every scale is in particular a $\fb$-scale).
Thus, their argument does not solve Hurewicz's Problem.

Just--Miller--Scheepers--Szeptycki, and independently Bartoszy\'nski--Shelah, solve Hurewicz's Conjecture in ZFC,
and thus improve Fremlin--Miller's result, but do not solve Hurewicz's Problem.

Chaber--Pol solve Hurewicz's Problem, and thus improve Fremlin--Miller's result, but do obtain Just--Miller--Scheepers--Szeptycki's theorem.

We have proved here that Bartoszy\'nski--Shelah's construction actually satisfies $\mathsf{U}_n(\Ga,\Ga)$,
and presented an improved construction satisfying $\mathsf{U}_2(\Ga,\Ga)$.
To get the stronger property $\mathsf{U}_1(\Ga,\Ga)$ (which is $\sone(\Ga,\Ga)$), axioms beyond
ZFC become necessary (Miller--Tsaban).
\forgotten

\subsection*{Acknowledgments}
We thank Gabor Lukacs, Lyubomyr Zdomskyy and the referee for their useful comments,
which lead to improvements in the presentation of this paper.

\ed
\begin{thebibliography}{99}

\bibitem{coc8}
L. Babinkostova, L. Ko\v{c}inac, and M. Scheepers,
\emph{Combinatorics of open covers (VIII)},
Topology and its Applications \textbf{140} (2004), 15--32.

\bibitem{BaSh01}
T. Bartoszy\'nski and S. Shelah,
\emph{Continuous images of sets of reals},
Topology and its Applications \textbf{116} (2001), 243--253.

\bibitem{ideals}
T. Bartoszy\'nski and B. Tsaban,
\emph{Hereditary topological diagonalizations and the Menger--Hurewicz Conjectures},
Proceedings of the American Mathematical Society \textbf{134} (2006), 605--615.

\bibitem{BlassHBK}
A. Blass,
\emph{Combinatorial cardinal characteristics of the continuum},
in: \textbf{Handbook of Set Theory} (M.\ Foreman, A.\ Kanamori, and M.\ Magidor, eds.),
Kluwer Academic Publishers, Dordrecht, to appear.
\texttt{http://www.math.lsa.umich.edu/\~{}ablass/hbk.pdf}

\bibitem{BukCie04}
L. Bukovsk\'y and K. Ciesielski,
\emph{Spaces on which every pointwise convergent series of continuous functions converges pseudo-normally},
Proceedings of the American Mathematical Society \textbf{133} (2004), 605--611.

\bibitem{ChaPol}
J.\ Chaber and R.\ Pol,
\emph{A remark on Fremlin--Miller theorem concerning the Menger property and Michael concentrated sets},
unpublished note (October 2002).

\bibitem{FM88}
D. Fremlin and A. Miller,
\emph{On some properties of Hurewicz, Menger and Rothberger},
Fundamenta Mathematica \textbf{129} (1988), 17--33.

\bibitem{GM}
F. Galvin and A. Miller,
\emph{$\gamma$-sets and other singular sets of real numbers},
Topology and its Applications \textbf{17} (1984), 145--155.

\bibitem{GFTM95}
S. Garcia-Ferreira and A. Tamariz-Mascarua,
\emph{Some generalizations of rapid ultrafilters and Id-fan tightness},
Tsukuba Journal of Mathematics \textbf{19} (1995), 173--185.

\bibitem{Hure25}
W. Hurewicz,
\emph{\"Uber eine Verallgemeinerung des Borelschen Theorems},
Mathematische Zeitschrift \textbf{24} (1925), 401--421.

\bibitem{Hure27}
W. Hurewicz,
\emph{\"Uber Folgen stetiger Funktionen},
Fundamenta Mathematicae \textbf{9} (1927), 193--204.

\bibitem{coc2}
W. Just, A. Miller, M. Scheepers, and P. Szeptycki,
\emph{The combinatorics of open covers II},
Topology and its Applications \textbf{73} (1996), 241--266.

\bibitem{coc7}
L. Ko\v{c}inac and M. Scheepers,
\emph{Combinatorics of open covers (VII): Groupability},
Fundamenta Mathematicae \textbf{179} (2003), 131--155.

\bibitem{Lelek}
A. Lelek, \emph{Some cover properties of spaces},
Fundamenta Mathematicae \textbf{64} (1964), 209--218.

\bibitem{Menger24}
K. Menger,
\emph{Einige \"Uberdeckungss\"atze der Punktmengenlehre},
Sitzungsberichte der Wiener Akademie \textbf{133} (1924), 421--444.

\bibitem{Mild01}
H.\ Mildenberger,
\emph{Groupwise dense families},
Archive for Mathematical Logic \textbf{40} (2001), 93--112.

\bibitem{BBC}
A. Miller and B. Tsaban,
\emph{Point-cofinite covers in Laver's model},
Proceedings of the American Mathematical Society \textbf{138} (2010), 3313--3321.

\bibitem{HurPerf}
D. Repov\v{s}, B. Tsaban, and L. Zdomskyy,
\emph{Hurewicz sets of reals without perfect subsets},
Proceedings of the American Mathematical Society \textbf{136} (2008), 2515--2520.

\bibitem{Sakai07}
M.\ Sakai,
\emph{The sequence selection properties of $C_p(X)$},
Topology and its Applications \textbf{154} (2007), 552--560.

\bibitem{GlCovs}
N. Samet, M. Scheepers, and B. Tsaban,
\emph{Partition relations for Hurewicz-type selection hypotheses},
Topology and its Applications \textbf{156} (2009), 616--623.

\bibitem{coc1}
M. Scheepers,
\emph{Combinatorics of open covers I: Ramsey theory},
Topology and its Applications \textbf{69} (1996), 31--62.

\bibitem{alpha_i}
M. Scheepers,
\emph{$C_p(X)$ and Arhangel'ski\u{\i}'s $\alpha_i$ spaces},
Topology and its Applications \textbf{89} (1998), 265--275.

\bibitem{strongdiags}
B. Tsaban,
\emph{Strong $\gamma$-sets and other singular spaces},
Topology and its Applications \textbf{153} (2005), 620--639.

\bibitem{sfh}
B. Tsaban and L. Zdomskyy,
\emph{Scales, fields, and a problem of Hurewicz},
Journal of the European Mathematical Society \textbf{10} (2008), 837--866.

\bibitem{Val98a}
J. Valueva,
\emph{On some fan-tightness type properties},
Commentationes Mathematicae Universitatis Carolinae \textbf{39} (1998), 415--421.

\bibitem{Val98}
J. Valueva,
\emph{A remark on combinatorics of open covers and $C_p$-spaces},
Questions and Ansers in General Topology \textbf{16} (1998), 183--187.

\end{thebibliography}
